\documentclass[12pt]{article}
\usepackage{amsmath,amsfonts,amssymb,amsthm,latexsym}
\theoremstyle{plain}
\newtheorem{theorem}{Theorem}

\newtheorem{lemma}[theorem]{Lemma}

\theoremstyle{definition}

\theoremstyle{remark}

\begin{document}

\begin{center}
\textbf{\large Fibonacci and Lucas Numbers Associated with Brocard-Ramanujan Equation}
\end{center}

\begin{center}
Prapanpong Pongsriiam

Department of Mathematics, Faculty of Science\\ Silpakorn University\\ Nakhon Pathom 73000, Thailand

prapanpong@gmail.com, pongsriiam\_p@silpakorn.edu
\end{center}

\begin{abstract}
We explicitly solve the diophantine equations of the form 
$$
A_{n_1}A_{n_2}\cdots A_{n_k}\pm 1 = B_m^2
$$
where $(A_n)_{n\geq 0}$ and $(B_m)_{m\geq 0}$ are either the Fibonacci sequence or Lucas sequence. This extends the result of D. Marques (2011) and L. Szalay (2012) concerning a variant of Brocard-Ramanujan equation. This is a manuscript of the article published in Communications of the Korean Mathematical Society, 32(3) (2017), pp. 511-522
\end{abstract}

2010 Mathematics Subject Classification. Primary 11B39; Secondary 11D99

Key words and Phases. Fibonacci number, Lucas number, Brocard-Ramanujan equation, Diophantine equation

\section{Introduction}

Let $(F_n)_{n\geq 0}$ be the Fibonacci sequence given by $F_0 = 0$, $F_1 = 1$, and $F_n = F_{n-1}+F_{n-2}$ for $n\geq 2$, and let $(L_n)_{n\geq 0}$ be the Lucas sequence given by the same recursive pattern as the Fibonacci sequence but with the initial values $L_0 = 2$ and $L_1 = 1$. The problem of finding all integral solutions to the diophantine equation 
\begin{equation}\label{eq1}
n!+1 = m^2
\end{equation}
is known as Brocard-Ramanujan problem. The known solutions to \eqref{eq1} are $(n,m) = (4,5), (5,11)$, and $(7,71)$ and it is still open whether the Brocard-Ramanujan equation has a solution when $n\geq 8$. Some variations of \eqref{eq1} have been considered by various authors and we refer the reader to \cite{BGa, Dab, DUl, Luc} and references therein for additional information and history. 

Marques \cite{Mar} considered a variant of \eqref{eq1} by replacing $n!$ by the product of consecutive Fibonacci numbers and $m^2$ by a square of a Fibonacci number. He claimed that the diophantine equation
\begin{equation}\label{eq1.2}
F_nF_{n+1}\cdots F_{n+k-1} + 1 = F_m^2
\end{equation}
 has no solution in positive integers $k$, $m$, $n$. But this is wrong, for example, $F_4+1 = F_3^2$ and $F_6+1 = F_4^2$ give solutions to the above equation. Szalay \cite[Theorem 2.1]{Sza} gives a correct version of Marques's result and considers the equations more general than \eqref{eq1.2}.

In this article, we continue the investigation by solving the following diophantine equations:
\begin{equation}\label{BReq1}
F_{n_1}F_{n_2}\cdots F_{n_k} \pm 1 = F_m^2,
\end{equation}
\begin{equation}\label{BReq2}
L_{n_1}L_{n_2}\cdots L_{n_k} \pm 1 = L_m^2,
\end{equation}
\begin{equation}\label{BReq3}
F_{n_1}F_{n_2}\cdots F_{n_k} \pm 1 = L_m^2,
\end{equation}
\begin{equation}\label{BReq4}
L_{n_1}L_{n_2}\cdots L_{n_k} \pm 1 = F_m^2,
\end{equation}
where $m\geq 0$, $k\geq 1$, $0\leq n_1 \leq n_2 \leq \cdots \leq n_k$. Note that unlike Marques \cite{Mar} and Szalay \cite{Sza}, we do not require $n_1, n_2, \ldots, n_k$ to be distinct. So \eqref{BReq1}, \eqref{BReq2}, \eqref{BReq3}, and \eqref{BReq4} are actually equivalent to, respectively, 
\begin{align*}
F_{n_1}^{a_1}F_{n_2}^{a_2}\cdots F_{n_\ell}^{a_\ell} \pm 1 &= F_m^2\\
L_{n_1}^{a_1}L_{n_2}^{a_2}\cdots L_{n_\ell}^{a_\ell} \pm 1 &= L_m^2\\
F_{n_1}^{a_1}F_{n_2}^{a_2}\cdots F_{n_\ell}^{a_\ell} \pm 1 &= L_m^2\\
L_{n_1}^{a_1}L_{n_2}^{a_2}\cdots L_{n_\ell}^{a_\ell} \pm 1 &= F_m^2
\end{align*}  
where $m\geq 0$, $\ell\geq 1$, $0\leq n_1 < n_2 < \cdots < n_\ell$, and $a_1, a_2, \ldots, a_\ell \geq 1$. For convenience, we sometimes go back and forth between the equations given in \eqref{BReq1} to \eqref{BReq4} and those which are equivalent to them such as the above ones. 

Note that Szalay \cite[Theorem 3.2]{Sza} considers the equation 
\begin{equation}\label{BReq5}
L_{n_1}L_{n_2}\cdots L_{n_k} + 1 = L_m^2
\end{equation} 
in non-negative integers $n_1 < n_2 < \cdots < n_k$, but it seems that he actually skips zero and thus missing the solution given by $L_0L_3 + 1 = L_2^2$. We give a correct version to this problem in Theorem \ref{thm3.103.6paper}. Finally, we remark that similar equations are also considered by Pongsriiam in \cite{Pon} and \cite{Pon1} where $F_m^2$ and $L_m^2$ in \eqref{BReq1}, \eqref{BReq2}, \eqref{BReq3}, and \eqref{BReq4} are replaced by $F_m$ and $L_m$, and where $\pm1$ is replaced by $0$.  

\section{Preliminaries and lemmas}

Since one of our main tools in solving the above equations is the primitive divisor theorem of Carmichael \cite{Car}, we first recall some facts about it. Let $\alpha$ and $\beta$ be algebraic numbers such that $\alpha+\beta$ and $\alpha\beta$ are nonzero coprime integers and $\alpha\beta^{-1}$ is not a root of unity. Let $(u_n)_{n\geq 0}$ be the sequence given by 
$$
\text{$u_0 = 0$, $u_1 = 1$, and $u_n = (\alpha+\beta)u_{n-1} - (\alpha\beta)u_{n-2}$ for $n\geq 2$.}
$$
Then we have Binet's formula for $u_n$ given by 
$$
u_n = \frac{\alpha^n-\beta^n}{\alpha-\beta}\quad\text{for $n\geq 0$.}
$$
So if $\alpha = \frac{1+\sqrt 5}{2}$ and $\beta = \frac{1-\sqrt 5}{2}$, then $(u_n)$ is the Fibonacci sequence.

A prime $p$ is said to be a primitive divisor of $u_n$ if $p\mid u_n$ but $p$ does not divide $u_1u_2\cdots u_{n-1}$. Then the primitive divisor theorem of Carmichael can be stated as follows.
\begin{theorem}\label{carthm} {\rm[Primitive divisor theorem of Carmichael \cite{Car}]}
Suppose $\alpha$ and $\beta$ are real numbers such that $\alpha+\beta$ and $\alpha\beta$ are nonzero coprime integers and $\alpha\beta^{-1}$ is not a root of unity. If $n\neq 1, 2, 6$, then $u_n$ has a primitive divisor except when $n = 12$, $\alpha+\beta = 1$ and $\alpha\beta = -1$. In particular, $F_n$ has a primitive divisor for every $n\neq 1, 2, 6, 12$ and $L_n$ has a primitive divisor for every $n\neq 1,6$.
\end{theorem}  
There is a long history about primitive divisors and the most remarkable results in this topic are given by Bilu, Hanrot, and Voutier  \cite{BHV}, by Stewart \cite{Ste}, and by Kunrui \cite{Kun}. For example, Bilu et al. \cite{BHV} extends Theorem \ref{carthm} to include the case where $\alpha$, $\beta$ are complex numbers. Nevertheless, Theorem \ref{carthm} is good enough in our situation.

Recall that we can define $F_n$ and $L_n$ for a negative integer $n$ by the formula 
$$
\text{$F_{-k} = (-1)^{k+1}F_k$ and $L_{-k} = (-1)^kL_k$ for $k\geq 0$.}
$$
Then we have the following identity which valid for all integers $m$, $k$. 
\begin{equation}\label{finaleq1}
F_{m-k}F_{m+k} = F_m^2 + (-1)^{m-k+1}F_k^2.
\end{equation}
The identity \eqref{finaleq1} can be proved using Binet's formula as follows:
\begin{align*}
F_{m-k}F_{m+k} &= \left(\frac{\alpha^{m-k}-\beta^{m-k}}{\alpha-\beta}\right)\left(\frac{\alpha^{m+k}-\beta^{m+k}}{\alpha-\beta}\right)\\
&= \frac{\left(\alpha^{m}-\beta^{m}\right)^2+2(\alpha\beta)^m - \left(\alpha^{m-k}\beta^{m+k}+\beta^{m-k}\alpha^{m+k}\right)}{(\alpha-\beta)^2}\\
&= F_m^2 + \frac{2(-1)^m-(-1)^{m-k}\left(\beta^{2k}+\alpha^{2k}\right)}{(\alpha-\beta)^2}\\
&= F_m^2 + \frac{2(-1)^m-(-1)^{m-k}\left((\beta^{k}-\alpha^{k})^2+2(-1)^k\right)}{(\alpha-\beta)^2}\\
&= F_m^2 + (-1)^{m-k+1}F_k^2.
\end{align*}
We will particularly apply \eqref{finaleq1} in the following form.
\begin{lemma}\label{lemma1}
For every $m\geq 1$, we have 
\begin{itemize}
\item[(i)] $F_m^2 - 1 = \begin{cases}
F_{m-1}F_{m+1}, &\text{if $m$ is odd};\\
F_{m-2}F_{m+2}, &\text{if $m$ is even}.
\end{cases}$
\item[(ii)] $F_m^2 + 1 = \begin{cases}
F_{m-1}F_{m+1}, &\text{if $m$ is even};\\
F_{m-2}F_{m+2}, &\text{if $m$ is odd}.
\end{cases}$
\end{itemize}
\end{lemma}
\begin{proof}
This follows from the substitution $k=1$ and $k=2$ in \eqref{finaleq1}. 
\end{proof}

We also need a factorization of $L_m^2\pm1$ as follows.
\begin{lemma}\label{lemma2}
For every $m\geq 1$, we have 
\begin{itemize}
\item[(i)] $L_m^2 - 1 = \begin{cases}
\displaystyle F_{3m}/F_{m}, &\text{if $m$ is even};\\
5F_{m-1}F_{m+1}, &\text{if $m$ is odd}.
\end{cases}$
\item[(ii)] $L_m^2 + 1 = \begin{cases}
\displaystyle F_{3m}/F_{m}, &\text{if $m$ is odd};\\
5F_{m-1}F_{m+1}, &\text{if $m$ is even}.
\end{cases}$
\end{itemize}
\end{lemma}
\begin{proof}
Similar to \eqref{finaleq1}, this can be checked easily using Binet's formula. 
\end{proof}

\section{Main results}

Consider the equations \eqref{BReq1}, \eqref{BReq2}, \eqref{BReq3}, and \eqref{BReq4}. Since $F_0 = 0$, $F_1 = F_2 = 1$, and $L_1 = 1$, we avoid some trivial solutions when $k\geq 2$ by assuming $3\leq n_1 \leq n_2 \leq \cdots \leq n_k$ in \eqref{BReq1} and \eqref{BReq3} and assuming $n_j\neq 1$ for every $j=1, 2, \ldots, k$ in \eqref{BReq2} and \eqref{BReq4}. In addition some parts of \eqref{BReq1} and \eqref{BReq2} are already considered by Szalay in \cite{Sza}, so we begin by giving the detailed proof for the solutions to \eqref{BReq3} and \eqref{BReq4}. Then we give a short discussion for \eqref{BReq2} and \eqref{BReq1}. 
\subsection{The equation $F_{n_1}F_{n_2}F_{n_3}\cdots F_{n_k}\pm 1 = L_m^2$}  
\begin{theorem}\label{thm1}
The diophantine equation
\begin{equation}\label{thmdioeq1}
F_{n_1}F_{n_2}F_{n_3}\cdots F_{n_k}+1 = L_m^2
\end{equation}
with $m\geq 0$, $k\geq 1$, and $0\leq n_1\leq n_2\leq \cdots \leq n_k$ has a solution if and only if $m = 0, 2, 4$ or $m$ is odd. In these cases, the nontrivial solutions to \eqref{thmdioeq1} are given by 
\begin{align*}
F_4&+1 = L_0^2,\quad F_0+1 = L_1^2, \quad F_3^3+1 = F_6+1 = L_2^2,\quad F_4F_5+1 = L_3^2, \\
F_3^4&F_4+1 = F_3F_4F_6+1 = L_4^2,  \quad F_4F_5F_6 + 1 = F_3^3F_4F_5+1 = L_5^2,  \\
F_3^3&F_5F_8 + 1 = L_7^2,\quad F_3^4F_4^2F_5F_{10}+1 = F_3F_4^2F_5F_6F_{10}+1 = L_{11}^2,\\
F_3^4&F_4^2F_5F_{14}+1 = F_3F_4^2F_5F_6F_{14}+1 = L_{13}^2,
\end{align*}
and an infinite family of solutions: $F_5F_{m-1}F_{m+1}+1 = L_m^2$ for any odd number $m\geq 7$. Here nontrivial solutions means that either $k=1$ or $k\geq 2$ and $n_1\geq 3$.
\end{theorem}
\begin{proof}
\textbf{Case 1} $m$ is even. Suppose for a contradiction that there exists $m\geq 5$ satisfying \eqref{thmdioeq1}. By Lemma \ref{lemma2}(i), we can write \eqref{thmdioeq1} as
\begin{equation}\label{thmdioeq2}
F_{n_1}F_{n_2}F_{n_3}\cdots F_{n_k}F_m = F_{3m}.
\end{equation}
By Theorem \ref{carthm}, if $3m > n_k$, then there exists a prime $p$ dividing $F_{3m}$ but does not divide any term on the left hand side of \eqref{thmdioeq2}. Similarly, if $3m < n_k$, there exists a prime $p\mid F_{n_k}$ but $p\nmid F_{3m}$, which is not the case. Hence $3m = n_k$. We remark that this kind of argument will be used repeatedly throughout the rest of this article. Then \eqref{thmdioeq2} is reduced to 
$$
F_{n_1}F_{n_2}F_{n_3}\cdots F_{n_{k-1}}F_m = 1.
$$
Then $1 = F_{n_1}F_{n_2}F_{n_3}\cdots F_{n_{k-1}}F_m \geq F_m \geq F_5 \geq 5$, which is a contradiction. Therefore $m\leq 4$. Now it is straightforward to check all values of $L_m^2-1$ for $m = 0, 2, 4$ and write it as a product of Fibonacci numbers. This leads to the solutions given by
$$
F_4+1 = L_0^2, \quad F_3^3+1 = F_6+1 = L_2^2, \quad F_3^4F_4+1 = F_3F_4F_6+1 = L_4^2.
$$
\noindent \textbf{Case 2} $m$ is odd. Then by Lemma \ref{lemma2}(i), we can write \eqref{thmdioeq1} as 
\begin{equation}\label{thmdioeq3}
F_{n_1}F_{n_2}F_{n_3}\cdots F_{n_k} = 5F_{m-1}F_{m+1}.
\end{equation}
Suppose first that $m\geq 14$. Then by Theorem \ref{carthm} and the same argument used in Case 1, we have $m+1 = n_k$ and \eqref{thmdioeq3} is reduced to 
\begin{equation}\label{thmdioeq3.55}
F_{n_1}F_{n_2}F_{n_3}\cdots F_{n_{k-1}} = 5F_{m-1}.
\end{equation}
This implies $k\geq 2$. Again by Theorem \ref{carthm}, $m-1 = n_{k-1}$ and \eqref{thmdioeq3.55} becomes 
$$
F_{n_1}F_{n_2}F_{n_3}\cdots F_{n_{k-2}} = 5.
$$
This implies that $k=3$ and $F_{n_1} = 5 = F_5$. In this case, we obtain an infinite number of solutions given by  
\begin{equation}\label{thmdioeq3.5}
F_5F_{m-1}F_{m+1} + 1 = L_m^2\quad\text{with $m\geq 14$ and $m$ is odd}.
\end{equation}
By Lemma \ref{lemma2}(i), we see that \eqref{thmdioeq3.5} also holds for any odd number $m\geq 3$. So we only need to check for the other factorizations of $L_m^2-1$ ($m$ odd and $m\leq 15$) as product of Fibonacci numbers. This leads to the other solutions to \eqref{thmdioeq1} as given in the statement of the theorem. 
\end{proof}
\begin{theorem}\label{thm6}
The diophantine equation
\begin{equation}\label{thm6eq1}
F_{n_1}F_{n_2}F_{n_3}\cdots F_{n_k} - 1 = L_m^2
\end{equation}
with $m\geq 0$, $k\geq 1$, and $0\leq n_1 \leq n_2 \leq \cdots \leq n_k$ has a solution if and only if $m=1$ or $m$ is even. In these cases, the nontrivial solutions to \eqref{thm6eq1} are given by 
$$
F_3-1 = L_1^2,\quad F_5-1 = L_0^2,\quad  F_3F_5-1 = L_2^2, \quad F_3F_5F_5-1 = L_4^2,
$$
 and an infinite famility of solutions
$$
\text{$F_5F_{m-1}F_{m+1}-1 = L_m^2$ for every even number $m\geq 6$.}
$$ 
Here nontrivial solutions means that either $k=1$ or $k\geq 2$ and $n_1\geq 3$. 
\end{theorem}
\begin{proof}
The proof of this theorem is very similar to that of Theorem \ref{thm1}. So we only give a brief discussion. If $m$ is odd, then we apply Lemma \ref{lemma2}(ii) to write \eqref{thm6eq1} as
$$
F_{n_1}F_{n_2}F_{n_3}\cdots F_{n_k}F_m = F_{3m}.
$$
From this point, we can follow the proof of Case 1 in Theorem \ref{thm1} and obtain the solutions given by $F_3-1 = L_1^2$. If $m$ is even, we apply Lemma \ref{lemma2}(ii) to write \eqref{thm6eq1} as $F_{n_1}F_{n_2}F_{n_3}\cdots F_{n_k} = 5F_{m-1}F_{m+1}$. Then we follow the proof of Case 2 in Theorem \ref{thm1} to obtain the desired result. 
\end{proof}

\subsection{The equation $L_{n_1}L_{n_2}L_{n_3}\cdots L_{n_k}\pm 1 = F_m^2$}  

\begin{theorem}\label{thm2.4lm=f}
The diophantine equation
\begin{equation}\label{thmdioeq4}
L_{n_1}L_{n_2}L_{n_3}\cdots L_{n_k}+1 = F_m^2
\end{equation}
with $m\geq 0$, $k\geq 1$, and $0\leq n_1\leq n_2 \leq \cdots \leq n_k$ has a solution if and only if $3\leq m\leq 7$, $m=10$, or $m = 14$. More precisely, the nontrivial solutions to \eqref{thmdioeq4} are given by 
\begin{align*}
L_2&+1 = F_3^2,\quad L_0^3+1 = L_0L_3+1 = F_4^2,\quad L_0^3L_2+1 = L_0L_2L_3+1 = F_5^2,\\
L_2^2&L_4+1 = F_6^2,\quad L_0^3L_2L_4+1 = L_0L_2L_3L_4+1 = F_7^2, \quad A+1 = F_{10}^2, \\
A&L_8 +1 = F_{14}^2,
\end{align*}
where $A = L_0^4L_2^3L_4 = L_0^3L_2L_4L_6 = L_0^2L_2^3L_3L_4 = L_0L_2L_3L_4L_6 = L_2^3L_3^2L_4$. Here nontrivial solutions means that either $k=1$ or $k\geq 2$ and $n_j\neq 1$ for any $j \in\{1, 2, \ldots, k\}$.
\end{theorem}
\begin{proof}
By Lemma \ref{lemma1}(i), we can rewrite \eqref{thmdioeq4} as 
\begin{equation}\label{newthm3.3eq1}
L_{n_1}L_{n_2}L_{n_3}\cdots L_{n_k} = F_aF_b
\end{equation}
where $a, b\in \{m-1, m+1\}$ or $a, b\in \{m-2,m+2\}$. Suppose that every $n_1, n_2, \ldots, n_k$ is zero. Then \eqref{newthm3.3eq1} becomes
\begin{equation}\label{newthm3.3eq2}
2^k = F_aF_b.
\end{equation}
By Theorem \ref{carthm} and the fact that $2\mid F_3$ and $3\mid F_{12}$, we see that $F_n$ has a prime divisor distinct from $2$ for every $n\neq 1,2,3,6$. So \eqref{newthm3.3eq2} implies that $a,b\in \{1,2,3,6\}$. Checking all possible choices, we see that the only solutions to \eqref{thmdioeq4} in this case is given by $L_0^3+1 = F_4^2$. Since $L_1 = 1$, we easily see that the case $n_j = 1$ for every $j$ does not give a solution. Similarly, $m=0,1,2$ does not lead to a solution. From this point on, we assume that there exists $j \in \{1, 2, \ldots, k\}$ such that $n_j\geq 2$, $n_i\neq 1$ for any $i$, and $m\geq 3$. Let $n_\ell$ be the smallest positive integer among $n_1, n_2, \ldots, n_k$. So $n_\ell \geq 2$ and $n_1, n_2, \ldots, n_{\ell-1} = 0$. 

\noindent\textbf{Case 1} $m$ is odd and $m\geq 27$. By Lemma \ref{lemma1}(i) and the identity $F_{2n} = F_nL_n$, which holds for $n\geq 1$, we can write \eqref{thmdioeq4} as 
\begin{equation}\label{thmdioeq5}
2^{\ell-1}\frac{F_{2n_\ell}}{F_{n_\ell}}\frac{F_{2n_{\ell+1}}}{F_{n_{\ell+1}}}\cdots \frac{F_{2n_k}}{F_{n_k}} = F_{m-1}F_{m+1}.
\end{equation}
By Theorem \ref{carthm}, we obtain $m+1 = 2n_k$ and \eqref{thmdioeq5} is reduced to 
\begin{equation}\label{thmdioeq6}
2^{\ell-1}\frac{F_{2n_\ell}}{F_{n_\ell}}\frac{F_{2n_{\ell+1}}}{F_{n_{\ell+1}}}\cdots \frac{F_{2n_{k-1}}}{F_{n_{k-1}}} = F_{n_k}F_{m-1}.
\end{equation} 
Note that if $m-1 < 2n_{k-1}$, then $n_k = \frac{m+1}{2} < m-1 < 2n_{k-1}$. So by applying Theorem \ref{carthm} to \eqref{thmdioeq6}, we obtain $m-1 = 2n_{k-1}$ and \eqref{thmdioeq6} is reduced to
\begin{equation}\label{thmdioeq6.5}
2^{\ell-1}\frac{F_{2n_\ell}}{F_{n_\ell}}\frac{F_{2n_{\ell+1}}}{F_{n_{\ell+1}}}\cdots \frac{F_{2n_{k-2}}}{F_{n_{k-2}}} = F_{n_{k-1}}F_{n_k}.
\end{equation}
Since $n_k\geq n_{k-1} = \frac{m-1}{2}\geq 13$, we can apply Theorem \ref{carthm} to \eqref{thmdioeq6.5} and repeat the above argument to obtain 
$$
\text{$n_k = 2n_{k-2}$ and $n_{k-1} = 2n_{k-3}$.}
$$
Then $m+1 = 2n_k = 4n_{k-2}$ and $m-1 = 2n_{k-1} = 4n_{k-3}$, and therefore $m+1$ and $m-1$ are divisible by $4$. So $4\mid (m+1)-(m-1) = 2$, a contradiction. Hence there is no solution in this case.

\noindent \textbf{Case 2} $m$ is even and $m\geq 54$. This case is similar to Case 1. We apply Lemma \ref{lemma1}(i) and the identity $F_{2n} = F_nL_n$ to write \eqref{thmdioeq4} in the form
\begin{equation}\label{thmdioeq7}
2^{\ell-1}\frac{F_{2n_\ell}}{F_{n_\ell}}\frac{F_{2n_{\ell+1}}}{F_{n_{\ell+1}}}\cdots \frac{F_{2n_k}}{F_{n_k}} = F_{m-2}F_{m+2}.
\end{equation}
Then we apply Theorem \ref{carthm} repeatedly to obtain 
\begin{align*}
m+2 &= 2n_k,\quad m-2 = 2n_{k-1}, \quad n_{k} = 2n_{k-2},\\
n_{k-1} &= 2n_{k-3}, \quad n_{k-2} = 2n_{k-4},\quad \text{and $n_{k-3} = 2n_{k-5}$}.
\end{align*}
Note that we can repeat this process as long as the indices of the Fibonacci numbers appearing on the right hand side of the equation are larger than 12. Here $n_{k-3} = \frac{n_{k-1}}{2} = \frac{m-2}{4} \geq 13$. So the above argument is justified. This leads to 
\begin{align*}
m+2 &= 2n_{k} = 4n_{k-2} = 8n_{k-4}\;\; \text{and}\\
m-2 &= 2n_{k-1} = 4n_{k-3} = 8n_{k-5}.
\end{align*}
Therefore $8\mid m+2$ and $8\mid m-2$. So $8\mid (m+2)-(m-2) = 4$, a contradiction. So there is no solution in this case.

\noindent From Case 1 and Case 2, we only need to consider the following:
\begin{equation}\label{thmdioeq7.1}
\text{$m$ is odd and $3\leq m\leq 25$},
\end{equation}
\begin{equation}\label{thmdioeq7.2}
\text{$m$ is even and $3\leq m\leq 52$}.
\end{equation}
Since $L_{n_k} \leq L_{n_1}L_{n_2}\cdots L_{n_k} = F_m^2-1 \leq L_{2m}$, we have $n_k \leq 2m \leq 104$. In addition, $2^k \leq L_{n_1}^k \leq L_{n_1}L_{n_2}\cdots L_{n_k} = F_m^2-1 \leq F_{52}^2-1$. So $k\leq \frac{\log (F_{52}^2-1)}{\log 2}$. So we only need to find the solutions to \eqref{thmdioeq4} in the range $1\leq k \leq \frac{\log (F_{52}^2-1)}{\log 2}$, $3\leq m \leq 52$, and $0 \leq n_1 \leq n_2 \leq \cdots \leq n_k\leq 104$. Since this is only a finite number of cases, it can be verified using computer programming. However, we think that checking it by hand does not take too much time. So we offer here a proof which does not require a high technology in computer programming.

Recall that for each positive integer $n$, the order of appearance of $n$ in the Fibonacci sequence, denoted by $z(n)$, is the smallest positive integer $k$ such that $n\mid F_k$. It is a well known fact that if $p$ is an odd prime and $z(p)$ is odd, then $p\nmid L_n$ for any $n\geq 0$. We refer the reader to Lemma 2.1 of Ward \cite{War} for a proof and other theorems in \cite{War} for related results. Since $z(5) = 5$, $z(13) = 7$, and $z(17) = 9$ are odd, the Lucas numbers are not divisible by any of $5$, $13$, and $17$. Here the calculation of $z(p)$ (for $p = 5, 13, 17$) is straightforward or it can be looked up in the Fibonacci Tables compiled by Brother A. Brousseau and distributed online by the Fibonacci Association \cite{Fib}. 

Next it is easy to calculate the period of $F_m$ modulo $5$, and $F_m$ modulo $13$. Again, this can also be looked up in the Fibonacci Tables \cite{Fib}. Then we see that 
\begin{equation}\label{thmdioeq7.3}
\text{$5\mid F_m^2-1$ when $m\equiv 1, 2, 8, 9, 11, 12, 18, 19\pmod{20}$}
\end{equation}
and
\begin{equation}\label{thmdioeq7.4}
\text{$13\mid F_m^2-1$ when $m\equiv 1, 2, 12, 13, 15, 16, 26, 27 \pmod{28}$}.
\end{equation}
Since Lucas numbers are not divisible by either $5$ or $13$, we see that if $m$ is in the residue classes given in \eqref{thmdioeq7.3} or \eqref{thmdioeq7.4}, then $F_m^2-1$ is not a product of Lucas numbers. Similarly, $F_m^2-1$ is not a product of Lucas numbers if $17\mid F_m^2-1$ which occurs when 
\begin{equation}\label{thmdioeq7.5}
m\equiv 1, 2, 16, 17, 19, 20, 34, 35\pmod{36}.
\end{equation}
So we eliminate those $m$ in \eqref{thmdioeq7.1} and \eqref{thmdioeq7.2} satisfying \eqref{thmdioeq7.3}, \eqref{thmdioeq7.4}, or \eqref{thmdioeq7.5}. At this point, we only need to consider $F_m^2-1$ in the following cases:
\begin{itemize}
\item[(i)] $3\leq m\leq 7$, $m = 10, 14$.
\item[(ii)] $23\leq m \leq 25$, $m = 36, 46, 50$.
\end{itemize}
By looking up the Fibonacci Tables \cite{Fib}, we see that $z(89) = 11$, $z(37) = 19$, $z(233) = 13$, which are odd numbers, and 
\begin{align*}
89&\mid F_m^2-1\quad\text{when $m = 23, 24, 46$},\\
37&\mid F_{m}^2-1\quad\text{when $m = 36$},\\
233&\mid F_m^2-1\quad \text{when $m = 25, 50$}.
\end{align*}
So those $m$ in (ii) does not give a solution to \eqref{thmdioeq4}. Now we only have a small number of $m$ in (i), which can be easily checked by hand. Each value of $m$ in (i) leads to a solution to \eqref{thmdioeq4}. This completes the proof. 
\end{proof}
\begin{theorem}
The diophantine equation
\begin{equation}\label{eqdiothm8}
L_{n_1}L_{n_2}L_{n_3}\cdots L_{n_k}-1 = F_m^2
\end{equation}
with $m\geq 0$, $k\geq 1$, and $0\leq n_1\leq n_2 \leq \cdots \leq n_k$ has a solution if and only if $0\leq m\leq 2$. In fact, the nontrivial solutions to \eqref{eqdiothm8} are given by $L_1-1 = F_0^2$, $L_0-1 = F_1^2$, and $L_0-1 = F_2^2$. Here nontrivial solutions means that either $k=1$ or $k\geq 2$ and $n_j\neq 1$ for any $j \in\{1, 2, \ldots, k\}$. 
\end{theorem}
\begin{proof}
The proof of this theorem is similar to that of Theorem \ref{thm2.4lm=f}. We first consider the case $n_1 = n_2= \cdots = n_k\in\{0,1\}$ and obtain the solutions given by $L_0-1 = F_1^2$, $L_0-1 = F_2^2$, and $L_1-1 = F_0^2$. Next if $m$ is even and $m \geq 12$, we follow the argument used in Case 1 of Theorem \ref{thm2.4lm=f} to write \eqref{eqdiothm8} as 
$$
2^{\ell-1}\frac{F_{2n_\ell}}{F_{n_\ell}}\frac{F_{2n_{\ell+1}}}{F_{n_{\ell+1}}}\cdots \frac{F_{2n_k}}{F_{n_k}} = F_{m-1}F_{m+1},
$$
where $\ell$ is defined in exactly the same way as that in Theorem \ref{thm2.4lm=f}.

Now the argument is a bit easier than that in Theorem \ref{thm2.4lm=f}. We see that Theorem \ref{carthm} forces $m+1 = 2n_k$, which contradicts the fact that $m$ is even. So there is no solution in this case. Similarly, there is no solution in the case that $m$ is odd and $m\geq 11$. Therefore we only need to consider the case $m\leq 10$. It is easy to check that
\begin{align*}
5\mid F_m^2+1 &\quad\text{if $m = 3,4,6,7$}\\
13\mid F_m^2+1 &\quad\text{if $m = 5,6,8,9$}\\
17\mid F_m^2+1 &\quad\text{if $m = 7,8,10$}.
\end{align*} 
Since Lucas numbers are not divisible by any of $5$, $13$, and $17$, we see that $F_m^2+1$ is not a product of Lucas numbers when $3\leq m\leq 10$. So we eliminate those $m$ and consider only $m = 0, 1, 2$ which lead to the solutions already obtained. This completes the proof. 
\end{proof}
\subsection{The equation $L_{n_1}L_{n_2}L_{n_3}\cdots L_{n_k}\pm 1 = L_m^2$}  
\begin{theorem}\label{thm3.9}
The diophantine equation
\begin{equation}\label{thm3.9eq1}
L_{n_1}L_{n_2}L_{n_3}\cdots L_{n_k}-1 = L_m^2
\end{equation}
with $m\geq 0$, $k\geq 1$, and $0\leq n_1\leq n_2 \leq \cdots \leq n_k$ has a solution if and only if $m=1$. The nontrivial solution to \eqref{thm3.9eq1} is given by $L_0-1 = L_1^2$. Here nontrivial solutions means that either $k=1$ or $k\geq 2$ and $n_j\neq 1$ for any $j \in \{1, 2, \ldots, k\}$.
\end{theorem}
\begin{proof}
We follow the argument in the proof of Theorem \ref{thm2.4lm=f} and let $\ell$ be defined in the same way. If $m$ is even, then by Lemma \ref{lemma2}(ii), we can write \eqref{thm3.9eq1} as
\begin{equation}\label{thm3.9eq3}
L_{n_1}L_{n_2}L_{n_3}\cdots L_{n_k} = 5F_{m-1}F_{m+1}.
\end{equation} 
Since $5$ does not divide any Lucas number, \eqref{thm3.9eq3} is impossible. So there is no solution in this case. Suppose $m$ is odd and $m\geq 5$. We apply Lemma \ref{lemma2}(ii) to write \eqref{thm3.9eq1} as 
\begin{equation}\label{thm3.9eq2}
2^{\ell-1}\frac{F_{2n_\ell}}{F_{n_\ell}}\frac{F_{2n_{\ell+1}}}{F_{n_{\ell+1}}} \cdots \frac{F_{2n_k}}{F_{n_k}}F_m = F_{3m}.
\end{equation}
Then from \eqref{thm3.9eq2} and Theorem \ref{carthm}, we obtain $3m = 2n_k$, which contradicts the fact that $m$ is odd. Therefore we only need to consider $m = 1, 3$ which can be easily checked. So the proof is complete. 
\end{proof}
\begin{theorem}\label{thm3.103.6paper}
The diophantine equation
\begin{equation}\label{thm3.10eq1}
L_{n_1}L_{n_2}L_{n_3}\cdots L_{n_k} + 1 = L_m^2
\end{equation}
with $m\geq 0$, $k\geq 1$, and $0\leq n_1\leq n_2 \leq \cdots \leq n_k$ has a solution if and only if $m = 0, 2, 4$. The nontrivial solutions to \eqref{thm3.10eq1} are given by 
\begin{align*}
L_2&+1 = L_0^2,\quad L_0L_3+1 = L_0^3+1 = L_2^2,\\
L_0^4&L_2+1 = L_0^2L_2L_3+1 = L_2L_3^2+1 = L_4^2.
\end{align*}
Here nontrivial solutions means that either $k=1$ or $k\geq 2$ and $n_j\neq 1$ for any $j \in \{1, 2, \ldots, k\}$.
\end{theorem}
\begin{proof}
We still follow the argument used in the proof of Theorem \ref{thm2.4lm=f} and let $\ell$ be defined in the same way. If $m$ is odd, then we apply Lemma \ref{lemma2}(i) to write \eqref{thm3.10eq1} as
\begin{equation}\label{thm3.10eq2}
L_{n_1}L_{n_2}L_{n_3}\cdots L_{n_k} = 5F_{m-1}F_{m+1}.
\end{equation}
Since $5$ does not divide any Lucas number, \eqref{thm3.10eq2} is impossible. So there is no solution to \eqref{thm3.10eq1} in this case. Next assume that $m$ is even and $m\geq 14$. By Lemma \ref{lemma2}(i) and the identity $F_{2n} = F_nL_n$, we can write \eqref{thm3.10eq1} as
\begin{equation}\label{thm3.10eq3}
2^{\ell-1}\frac{F_{2n_\ell}}{F_{n_\ell}}\frac{F_{2n_{\ell+1}}}{F_{n_{\ell+1}}} \cdots \frac{F_{2n_k}}{F_{n_k}}F_m = F_{3m}.
\end{equation}
By Theorem \ref{carthm}, $3m = 2n_k$ and \eqref{thm3.10eq3} is reduced to  
\begin{equation}\label{thm3.10eq4}
2^{\ell-1}\frac{F_{2n_\ell}}{F_{n_\ell}}\frac{F_{2n_{\ell+1}}}{F_{n_{\ell+1}}} \cdots \frac{F_{2n_{k-1}}}{F_{n_{k-1}}}F_m = F_{n_k}.
\end{equation}
Since $n_k = \frac{3m}{2} > m \geq 14$, we obtain by Theorem \ref{carthm} that $n_k = 2n_{k-1}$ and \eqref{thm3.10eq4} is reduced to 
\begin{equation}\label{thm3.10eq5}
2^{\ell-1}\frac{F_{2n_\ell}}{F_{n_\ell}}\frac{F_{2n_{\ell+1}}}{F_{n_{\ell+1}}} \cdots \frac{F_{2n_{k-2}}}{F_{n_{k-2}}}F_m = F_{n_{k-1}}.
\end{equation}
Now $n_{k-1} = \frac{n_k}{2} = \frac{3m}{4} < m$, so $F_m$ has a primitive divisor which does not divide $F_{n_{k-1}}$. Therefore \eqref{thm3.10eq5} is impossible. Hence there is no solution in this case. So we only need to consider $m\leq 12$ and $m$ is even. This can be easily checked. So the proof is complete.  
\end{proof}

\subsection{The equation $F_{n_1}F_{n_2}F_{n_3}\cdots F_{n_k}\pm 1 = F_m^2$}

Following Szalay \cite{Sza}, we let
\begin{align*}
\varepsilon &= \varepsilon(m) = \begin{cases}
1, & \text{if $m$ is odd};\\
2, & \text{if $m$ is even},
\end{cases} \\
\delta &= \delta(m) = \begin{cases}
1, & \text{if $m$ is even};\\
2, & \text{if $m$ is odd}.
\end{cases}
\end{align*}
Then we have the following result.
\begin{theorem}\label{thm3.7}
The diophantine equation
\begin{equation}\label{thm3.7eq1}
F_{n_1}F_{n_2}F_{n_3}\cdots F_{n_k}-1 = F_m^2
\end{equation}
with $m\geq 0$, $k\geq 1$, and $0\leq n_1 \leq n_2 \leq \cdots \leq n_k$ has a solution for every $m\geq 0$. The nontrivial solutions to \eqref{thm3.7eq1} are given by 
$$
F_1-1 = F_2-1 = F_0^2,\quad F_3-1 = F_1^2, \quad F_3-1 = F_2^2, \quad F_5-1 = F_3^2,
$$
and an infinite family of solutions:
$$
F_{m-\delta}F_{m+\delta}-1 = F_m^2\quad\text{for all $m\geq 4$}.
$$
Here nontrivial solutions means that either $k=1$ or $k\geq 2$ and $n_1\geq 3$.
\end{theorem}
\begin{proof}
The proof of this theorem is similar to the others, so we only give a brief discussion. If $m$ is even and $m\geq 14$, we apply Lemma \ref{lemma1}(ii) to write \eqref{thm3.7eq1} as
$$
F_{n_1}F_{n_2}F_{n_3}\cdots F_{n_k} = F_{m-1}F_{m+1}.
$$
Applying Theorem \ref{carthm} repeatedly, we obtain $m+1 = n_k$, $m-1 = n_{k-1}$, and $k=2$. If $m$ is odd and $m\geq 15$, we apply Lemma \ref{lemma1}(ii) and follow the same argument to obtain $m+2 = n_k$, $m-2 = n_{k-1}$, and $k=2$. The case $m\leq 13$ can be checked by hand.  
\end{proof}
\begin{theorem}\label{thm8l2}
The diophantine equation
$$
F_{n_1}F_{n_2}F_{n_3}\cdots F_{n_k}+1 = F_m^2
$$
with $m\geq 0$, $k\geq 1$, and $0\leq n_1\leq n_2 \leq \cdots \leq n_k$ has a solution for every $m\geq 1$. The nontrivial solutions to the above equations are given by
\begin{align*}
F_0&+1 = F_1^2,\quad F_0+1 = F_2^2,\quad F_4+1 = F_3^2,\quad F_3^3+1 = F_6+1 = F_4^2,  \\
F_3^3&F_4+1 = F_5^2,\quad F_3^3F_8+1 = F_7^2,\quad F_3^3F_{10}+1 = F_8^2,\quad AF_8+1 = F_{10}^2,\\
AF&_{10}+1 = F_{11}^2, \quad AF_{14}+1 = F_{13}^2,\quad AF_{16}+1 = F_{14}^2,
\end{align*}
where $A = F_3F_4^2F_6 = F_3^4F_4^2$, and an infinite family of solutions:
$$
F_{m-\varepsilon}F_{m+\varepsilon} + 1 = F_m^2\quad\text{for all $m\geq 5$}.
$$
Here nontrivial solutions means that either $k=1$ or $k\geq 2$ and $n_1\geq 3$.
\end{theorem}
\begin{proof}
The proof of this theorem is similar to that of Theorem \ref{thm3.7}. The only difference is that we apply Lemma \ref{lemma1}(i) instead of Lemma \ref{lemma1}(ii). We leave the verification to the reader. 
\end{proof}
\noindent\textbf{Comments:} The author believes that his method can be used to solve other equations of this type where $(F_n)_{n\geq 1}$ and $(L_n)_{n\geq 1}$ are replaced by some general second order linear recurrence sequences. But the author will leave this problem to the interested reader. Nevertheless, he will consider another Fibonacci version of Brocard-Ramanujan equation in the next article.

\noindent \textbf{Acknowledgment} The author would like to thank Professor D\c{a}browski and Professor Szalay for sending him their articles (\cite{Dab}, \cite{DUl}, \cite{Sza}) which give him some ideas in writing this article. He also thanks the anonymous referee for carefully reading the manuscript and for his or her suggestions which help improve the presentation of this article.

\end{document}